\documentclass[11pt]{amsart}
\usepackage{euscript}
\usepackage{a4wide}
\usepackage[T2A]{fontenc}
\usepackage[utf8]{inputenc}
\usepackage[english]{babel}
\usepackage{amsfonts}
\usepackage{amssymb, amsthm}
\usepackage{amsmath}
\usepackage{graphicx}
\usepackage{geometry}

\renewcommand{\epsilon}{\varepsilon} 
\renewcommand{\phi}{\varphi} 

\def\ge{\geqslant}

\def\le{\leqslant}

\def\ol{\overline}
\def\kratno{\lower.5ex\hbox{$\,\vdots\,$}}

\def\eps{\epsilon}
\def\q#1.{\smallbreak\noindent\hskip15pt{\bf#1.}\enspace\ignorespaces} 
\def\dotline{\smallskip\hbox to \hsize{\dotfill}\medskip}

\def\norm[#1]{\left\| #1 \right\|}

\newcommand{\R}{\mathbb{R}}

\newcommand{\T}{\mathbb{T}}
\newcommand{\D}{\mathbb{D}}
\newcommand{\Cm}{\mathbb{C}}

\newcommand{\ilim}{\int\limits}
\newcommand{\slim}{\sum\limits}

\renewcommand{\Im}{\mathop{\mathrm{Im}}\nolimits}
\renewcommand{\Re}{\mathop{\mathrm{Re}}\nolimits}
\newcommand{\loc}{\mathop{\mathrm{loc}}\nolimits}

\theoremstyle{plain}
\newtheorem{thm}{Theorem} 
\newtheorem{knownthm}{Theorem} 
\newtheorem{corol}{Corollary}
\newtheorem*{lm}{Lemma}
\newtheorem{lemma}{Lemma}
\newtheorem{knownlemma}{Lemma} 
\newtheorem*{st}{Statement}

\newtheorem{conj}{Conjecture}
\newtheorem*{thrm}{Theorem}

\theoremstyle{definition}
\newtheorem*{defn}{Definition}

\theoremstyle{remark}
\newtheorem*{rem}{Remark}

\begin{document}
\newgeometry{top=20mm,bottom=20mm,left=35mm,right=20mm}
\renewcommand{\theknownthm}{\Alph{knownthm}}
\renewcommand{\theconj}{\Roman{conj}}
\renewcommand{\theknownlemma}{\Alph{knownlemma}}


\title{Mate-Nevai-Totik theorem for Krein systems}

\date{}
\author{Pavel Gubkin}
\address{
\begin{flushleft}
Pavel Gubkin: pasha\_gubkin\_v@mail.ru\\\vspace{0.1cm}
Leonhard Euler International Mathematical Institute\\
14th Line 29B, Vasilyevsky Island, St. Petersburg, 199178, Russia
\\\vspace{0.1cm}
St. Petersburg State University \\
Universitetskaya nab. 7-9, St. Petersburg, 199034, Russia
\end{flushleft}
}
\thanks{The work is supported by Ministry of Science and Higher Education of the Russian Federation, agreement № 075–15–2019–1619.}

\maketitle

\begin{abstract}
    We prove the Ces\`aro boundedness of eigenfunctions of the Dirac operator on the half-line with a square-summable potential. The proof is based on the theory of Krein systems and, in particular, on the continuous version of a theorem by A. Mate, P. Nevai and V. Totik from 1991. 
\end{abstract}
\section{Introduction}
\subsection{Main result}
Consider  a symmetric zero trace $2\times 2$ potential $Q = \left(\begin{smallmatrix}-q&p\\p&q\end{smallmatrix}\right)$ on $\R_+ = [0,+\infty)$ with real-valued entries. We say that $Q\in L^p$ if both functions $p,q\in L^p$. Consider the differential equation 
\begin{align}\label{differential_equation}
    Jf'(t) + Qf(t) &= \lambda f(t) ,\quad t\in\R_+,\quad f(0) = 
    \begin{pmatrix}
    1\\
    0
\end{pmatrix},
\end{align}
where $J = \left(\begin{smallmatrix}0&1\\-1&0\end{smallmatrix}\right)$ is a square root of the identity matrix. This equation can be read as eigenfunction equation for Dirac operator $\mathcal{D}_Q = J\frac{d}{dr} + Q$ on the half-line $\R_+$.

Motivated by scattering theory, M. Christ and A. Kiselev studied \cite{christ2002}, \cite{christ20012}, \cite{christ20011}  the one-dimensional Shr\"{o}dinger operator using a method of multilinear expansion; their results can as well be applied to Dirac-type operators. Among other things, they proved the existence of wave operators for the Dirac operator in the case when $Q\in L^p$, $1\le p < 2$, which implies that for such Q the absolutely continuous spectrum of $D_Q$ is the entire real line $\R$. Deift and Killip extended \cite{deift1999} the spectrum result for $Q\in L^2$. Using methods of the theory of Krein systems, Denisov showed \cite{denisov2004} the existence of wave operators in the case when $Q\in L^2$; recently Bessonov extended \cite{bessonov2020} this result for Dirac operators whose spectral measure belongs to the Szeg\H{o} class on the real line. Some related results and approaches can be found in \cite{muscalu2003}, \cite{muscalu20032}, \cite{simon1982}. 

M. Christ and A. Kiselev proved in Section $6.3$ in \cite{christ2002} that the solution of \eqref{differential_equation} is bounded for almost every $\lambda\in\R$ in the case when $Q\in L^p$, $1\le p < 2$. Related question was studied by C. Muscalu, T. Tao, and C. Thiele in \cite{muscalu2003}; authors  posed the following conjecture. 

\begin{conj}\label{muscalu_conj}
Assume that $Q\in L^2$. Then for almost every $\lambda \in\R$ the solution $f$ of the differential equation \eqref{differential_equation} is bounded on $\R_+$.
\end{conj}
In present paper we prove an ``averaged'' version of Conjecture \ref{muscalu_conj}.
\begin{thm}\label{cesaro_boundness_for_dirac}
Assume that $Q\in L^2$. Then for almost every $\lambda \in\R$ the solution $f$ of the differential  equation \eqref{differential_equation} is bounded on $\R_+$ in $L^2$-Ces\`aro sense, i.e.,
\begin{align*}
    \sup_{r > 0}\left(\frac{1}{r}\int_0^r |f(t)|^2\, dt \right) < \infty.
\end{align*}
\end{thm}
Proof of this theorem is based on the theory of Krein systems which plays a role of a bridge between the theory of orthogonal polynomials on the unit circle and spectral theory of second order differential self-adjoint operators. 
Krein systems  were first introduced by M. G. Krein in  \cite{Krein} and later developed by different authors (see, e.g.,  \cite{rybalko1966}, \cite{sakhnovich2000}, \cite{Denisov2002}, \cite{Teplyaev2005}, \cite{bessonov20201}). Detailed account of the theory can be found in \cite{Denisov2009}. 

\medskip Let us start with a description of basic objects and results in the theory of orthogonal polynomials and in the theory of the Krein systems.
\subsection{Orthogonal polynomials on the unit circle }
Orthogonal polynomials on the unit circle appear in many areas of mathematics. An account of the theory can be found in  books \cite{Szego} by G. Szeg\H{o} and \cite{Simon} by B. Simon. Let $\mu$ be a probability measure on the interval $[-\pi, \pi]$ such that the support of $\mu$ is not a finite set. Then there exists a family of orthonormal polynomials $\{\phi_n\}_{n\ge 0 }$ in $L^2( [\pi, \pi],\mu)$. More precisely, let the family $\phi_n$ be  defined by 
\begin{align}\label{OPUC_defenition}
    \deg \phi_n = n,\quad \langle \phi_n, \phi_m \rangle_{L^2(\mu)} = \int_{-\pi}^{\pi}\phi_n(e^{it})\ol{\phi_m(e^{it})}\,d\mu(t) = \delta_{n,m},
\end{align}
where $\delta_{n,m}$ is the Kronecker symbol. Define a family of reversed orthogonal  polynomials by 
\begin{align*}
    \phi_n^*(z) = z^n \ol{\phi_n(1/\ol{z})},\quad z\in\Cm.
\end{align*}
It is easy to see that $\phi_n^*$ is a polynomial of degree not greater that $n$.
We say that the measure $\mu$ belongs to the Szeg\H{o} class on the unit circle if \begin{align}\label{Szego_condition_discrete}
    \int_{-\pi}^{\pi} \log \mu'(\theta)\,d\theta > -\infty,
\end{align}
where $\mu'$ is the density of the absolutely continuous part of $\mu$ with respect to the Lebesgue measure on $[-\pi, \pi]$. The following theorem is a result of work of different authors, we refer to it as the Szeg\H{o} theorem. This statement is a combination of Theorem 2.4.1 and Theorem 2.7.15 from \cite{Simon}. 
\begin{knownthm}[Szeg\H{o} theorem]\label{Szego_theorem}
    Let $\mu$ be a nontrivial measure on $[-\pi, \pi]$ and let $\phi_n$ be the associated orthogonal polynomials. Then the following conditions are equivalent: 
    \begin{enumerate}
        \item[$(a)$] the measure $\mu$ belongs to the Szeg\H{o} class,
        \item[$(b)$] for some $z_0$ such that $|z_0| < 1$ we have $\slim_{n = 0}^{\infty}|\phi_n(z_0)|^2 <\infty$,
        \item[$(c)$] there exists an analytic function $\Pi$ in the unit disc $\mathbb{D} = \{z\in \Cm\colon |z| < 1\}$ such that
        \begin{align*}
            \Pi(z) = \lim_{n\to\infty} \phi_n^*(z),\quad z\in\D,
        \end{align*}
        and the convergence is uniform on compact subsets in $\D$.
    \end{enumerate}
\end{knownthm}
\noindent If the equivalent conditions of the Szeg\H{o} theorem hold, then the series in part $(b)$ converges uniformly on compact subsets of $\D$. Moreover,  the function $\Pi$ is an outer function in the unit disc (see Theorem 2.3.5 \cite{Simon}) and $|\Pi(e^{it})|^{-2} = \mu'(t)$ almost everywhere on $[-\pi, \pi]$, i.e,
\begin{align*}
    \Pi(z) &= \exp\left[-\frac{1}{4\pi}\int_{-\pi}^{\pi}\frac{e^{i\theta} + z}{e^{i\theta} - z}\log \mu'(\theta)\,d\theta\right],\quad z\in\D.
\end{align*}
The function $D = \Pi^{-1}$ is called the Szeg\H{o} function and often is used instead of $\Pi$.
\subsection{Krein systems}
Krein systems  were first considered by M. G. Krein \cite{Krein} in 1954. Solutions of the Krein systems have properties similar to properties of orthogonal polynomials on the unit circle. Through such similarities it becomes possible to apply methods and ideas from the theory of orthogonal polynomials on the unit circle to spectral problems for self-adjoint differential operators. Detailed account of this approach can be found in the paper \cite{Denisov2009} by S. Denisov.
\begin{defn}
Let $a $ be a complex-valued function on the half-line $\R_+$. The Krein system with the coefficient $a$ is the following system of differential equations:
\begin{align}\label{diff_system}
\begin{cases}
   \frac{\partial}{\partial r}P(r,\lambda) = i\lambda P(r,\lambda) - \ol{a(r)}P_*(r,\lambda),&\quad \,\,P(0,\lambda) = 1,\\
   \frac{\partial}{\partial r}P_*(r,\lambda) = - a(r) P(r, \lambda),&\quad P_*(0,\lambda) = 1.
\end{cases}
\end{align}
In present paper we consider the case when $a\in L^1_{\loc}(\R_+)$, that is, $a\in L^1[ 0,x]$ for every $x\ge  0$. Here the complex number $\lambda$ is a spectral parameter; solutions $P$ and  $P_*$ are called continuous analogs of orthogonal polynomials on the unit circle.  
\end{defn}
Solutions $P(r,\cdot)$ and $P_*(r,\cdot)$ are entire functions. The parameter $r\ge 0$ plays the same role as the index $n\ge 0$ in \eqref{OPUC_defenition}: $P(r,\cdot)$ has the exponential type $r$ and $P_*(r,\cdot)$ has the exponential type not greater than $r$, see Chapter 3 in \cite{Denisov2009}.
For any function $a\in L^1_{\loc}(\R_+)$ there exists a Borel measure $\sigma$ on the real line (see \cite{Denisov2009}) such that
\begin{align}\label{first_condition_Szego_class}
    \int_{\R}\frac{d\sigma(x)}{1 + x^2} < \infty,
\end{align}
and the map
\begin{align*}
    \mathcal{U}_{\sigma}:f\mapsto \int_0^{\infty}f(r)P(r,\lambda)\,dr,
\end{align*}
initially defined on simple measurable functions with compact support can be continued to an isometry from $L^2(\R_+)$ to $L^2(\R, \sigma)$. This measure is called the spectral measure of system \eqref{diff_system}.
We say that the measure $\sigma$ belongs to the Szeg\H{o} class on the real line if it satisfies condition \eqref{first_condition_Szego_class} and 
\begin{align}
    \int_{\R}\frac{|\log \sigma'(x)|}{1 + x^2}\, dx < \infty, \label{Szego_main_condition}
\end{align}
where $\sigma'$ is the density of the absolutely continuous part of $\sigma$ with respect to the Lebesgue measure on $\R$. A key result in the theory of the Krein systems in an analog of the Szeg\H{o} theorem. It was originally formulated by Krein in \cite{Krein} with minor inaccuracies subsequently corrected by A. Teplyaev in \cite{Teplyaev2005}. For a complete proof see Chapter 8 in \cite{Denisov2009}.
\begin{knownthm}[Krein theorem]\label{Krein_theorem}
    Let a measure $\sigma$ be the spectral measure of some Krein system \eqref{diff_system} and let $P, P_*$ be the solutions of that system. Then the following conditions are equivalent:
    \begin{enumerate}
        \item[$(a)$] the measure  $\sigma$ belongs to Szeg\H{o} class,
        \item[$(b)$] for a point $\lambda_0\in \Cm_+ = \{z\in\Cm\mid \Im(z) > 0\}$ we have
        \begin{align*}
            \int_0^{\infty}|P(r,\lambda_0)|^2 \, dr <\infty,
        \end{align*}
        \item[$(c)$] there exists analytic in  $\Cm_+$ function $\Pi$ and a sequence  $r_n\to+\infty$ such that
        \begin{align*}
            \Pi(\lambda) = \lim_{n\to\infty}P_*(r_n,\lambda),
        \end{align*}
        and the convergence is uniform on compact subsets in $\Cm_+$.
    \end{enumerate}
\end{knownthm}
If the equivalent conditions of the Krein theorem hold then the integral in part $(b)$ converges uniformly on compact subsets. Moreover, $\Pi$ is an outer function in the upper half-plane and       \begin{align}\label{pi_modulus_on_the_boundary}
         |\Pi(x)|^2 = |2\pi\sigma'(x)|^{-1}
\end{align}
almost everywhere on $\R$ (see Lemma 8.6 in \cite{Denisov2009}), i.e.,
\begin{align}\label{Pi_definition}
    \Pi(\lambda) = \frac{1}{\sqrt{2\pi}}\exp\left[\frac{1}{2\pi i}\ilim_{-\infty}^{\infty}\left(\frac{1}{s - \lambda} - \frac{s}{s^2 + 1}\right)\log \sigma'(s)\, ds\right],\quad\lambda\in\Cm_+.
\end{align}
In addition, the function $[(\lambda + i)\Pi(\lambda)]^{-1}$ belongs to the Hardy space $H^2$.
\subsection{Mate-Nevai-Totik theorem}
A. Mate, P. Nevai, V. Totik in 1991 proved \cite{Mate_Nevai_Totik} asymptotics of so-called Christoffel functions. Their result in the language of orthogonal polynomials can be considered as a refinement of part (c) of the Szeg\H{o} theorem for points on the unit circle $\T$. Let us formulate their result. Let $\mu$ be a probability Borel measure on $[-\pi, \pi]$. For $z\in\Cm$, define an $n$-th Christoffel function $w_n$ as follows:
\begin{align*}
    w_n(\mu, z) = \min\left\{\frac{1}{2\pi|P(z)|^2}\int_{-\pi}^{\pi}|P(e^{i\theta})|^2\,d\mu(\theta)\text{ } \bigg| \text{ } P\text{ -- polynomial, } \deg P < n, P(z)\neq 0 \right\}.
\end{align*}
\begin{knownthm}[A. Mate, P. Nevai, V. Totik Theorem 1, \cite{Mate_Nevai_Totik}]\label{MNT_main}
Assume that the measure $\mu$ belongs to the Szeg\H{o} class on the unit circle. Then we have
\begin{align*}
    \lim_{n\to\infty} n w_n(\mu, e^{it}) = \mu'(t)
\end{align*}
for almost every $t\in[-\pi,\pi]$.
\end{knownthm}
\noindent The function $w_n$ admits the representation
\begin{align}\label{w_n_representation}
    2\pi w_n(\mu, z) = \left(\slim_{k = 0}^{n - 1}|\phi_k( z)|^2\right)^{-1},\quad z\in\Cm,
\end{align}
where $\phi_n$ are the orthonormal polynomials on the unit circle  associated with $\mu$, see Theorem 11.3.1 in \cite{Szego}. This representation allows us to  formulate Theorem  \ref{MNT_main} using only the orthogonal polynomials.
\begin{knownthm}\label{MNT_convergence_for_OPUC}
    Let $\mu$ be a nontrivial probability Borel measure on $\T$ and $\phi_n$ be the orthogonal polynomials associated with $\mu$. If $\mu$ belongs to the Szeg\H{o} class, then for almost every $t\in[-\pi,\pi]$
    \begin{align*}
        \frac{1}{n}\slim_{k = 0}^{n - 1}|\phi_k(e^{it})|^2\to \frac{1}{ 2\pi \mu'(t)},\quad n\to\infty.
    \end{align*}
\end{knownthm}
The core result of the present paper is the following theorem. It can be seen as a refinement of part (c) of the Krein theorem for points on the real line.
\begin{thm}\label{MNT_main_theorem}
Consider a Krein system with a coefficient in $L^1_{\loc}(\R_+)$ such that its spectral measure $\sigma$ is in the Szeg\H{o} class on the real line. Then for almost every $x\in \R$ there exists the limit
\begin{align*}
    \frac{1}{r}\int_0^r |P(s,x)|^2ds\to\frac{1}{2\pi \sigma'(x)},\quad r\to\infty.
\end{align*}
\end{thm}
\subsection{Teplyaev's conjectures}
Notice that the convergence in part (c) of the Krein theorem holds only for some sequence of real numbers $r_n$ while there is no such constraint in the part (c) of the Szeg\H{o} theorem. A. Teplyaev proved \cite{Teplyaev2005}  that the convergence $P_*(r,\lambda)\to\Pi(\lambda)$  holds for $a\in L^2(\R_+)$, see Theorem 1 in \cite{Teplyaev2005} or Theorem 11.1 in \cite{Denisov2009}. Additionally, the author proved sharpness of this result in the sense that it cannot be extended to any $L^p$ for $p > 2$, see Theorem 3 in \cite{Teplyaev2005}.
Teplyaev conjectured two positive results concerning the convergence of $P_*$ to $\Pi$ for a Krein system with the real-valued coefficient, see Conjecture 6.5, Conjecture 6.6 in \cite{Teplyaev2005}. The first conjecture was proved by S. Denisov in \cite{Denisov2009}, we state it below:
\begin{knownthm}[Lemma 8.6 \cite{Denisov2009}]
    Assume that the conditions of the Krein theorem hold, the function $a$ is real-valued, and a sequence $t_n\to\infty $ is such that  $ P(t_n,\lambda_0)\to 0 $ for some $\lambda_0\in\Cm_+$. Then $ P_*(t_n,\lambda) \to \Pi(\lambda)$ uniformly on compact subsets in $\Cm_+$.
\end{knownthm}
The second conjecture concerns a convergence of $P_*$ in average (i.e. in the Ces\`aro sense).  We prove it in the present paper.
\begin{thm}\label{Teplyaev_average_convergence}
    Assume $a$ is a real-valued function and conditions of Krein theorem hold. Then $\Pi(\lambda)$ is the limit in average of $P_*(r,\lambda)$, that is,
    \begin{align*}
        \Pi(\lambda) = \lim_{r\to\infty}\frac{1}{r}\int_0^r P_*(\rho, \lambda)\, d\rho,
    \end{align*}
    and the convergence is uniform on compact subsets in  $\Cm_+$.
\end{thm}
Notice that the limit in average of $P_*(r,\lambda)$ does not exist in general. In other words, there exists a function $a\in L^1_{\loc}$ such that the convergence in Theorem \ref{Teplyaev_average_convergence} does not hold, see discussion in the end of Section 6 in \cite{Teplyaev2005} .  

\subsection{Open problems}
It is unknown to us whether the convergence in Theorem \ref{MNT_convergence_for_OPUC} or in Theorem \ref{MNT_main_theorem} can be improved. We formulate two related conjectures.
\begin{conj}
Assume that the measure $\mu$ on the unit circle belongs to the Szeg\H{o} class on the unit circle. Let $\phi_n$ be the orthogonal polynomials associated with $\mu$. Then for almost every $t\in [-\pi,\pi]$ we have
\begin{align*}
    \frac{1}{n}\sum_{k = 0}^{n - 1}\phi_k^*(e^{it}) \to \frac{1}{\sqrt{ 2\pi \mu'(t)}},\quad n\to\infty.
\end{align*}
\end{conj}
This may be connected with works \cite{aptekarev2015}, \cite{denisov2017}, \cite{denisov2018}.
\begin{conj}
Consider Krein system \eqref{diff_system}. Let $P, P_*$ be its solitions and let $\sigma$ be its spectral measure. Assume that $\sigma$ is in the Szeg\H{o} class on the real line. Then there exist a function $\xi\colon \R_+\to\Cm$ such that the convergence 
\begin{align*}
    \frac{1}{r}\int_0^r P(s,x)\xi(s)\, ds \to \frac{1}{\sqrt{2\pi \sigma'(x)}},\quad r\to\infty
\end{align*}
holds for almost every $x\in\R$. Additionally, the function $\xi$ can be chosen real-valued if the coefficient of the Krein system is real-valued.
\end{conj}
The latter conjecture may be related to the results of \cite{bessonov2020}.
\subsection{Structure of the paper}
In Section \ref{proof of MNT theorem} we provide some preliminaries on Krein systems and prove Theorem \ref{MNT_main_theorem}. Proof of Teplyaev conjecture, Theorem \ref{Teplyaev_average_convergence} is given in Section \ref{proof of Teplyaev conjecture}. In the next section we use Theorem \ref{MNT_main_theorem} to prove Theorem \ref{cesaro_boundness_for_dirac}. 

\section{Dirac equation. Proof of Theorem \ref{cesaro_boundness_for_dirac}}\label{proof of theorem 1}
Let us show how to derive Theorem \ref{cesaro_boundness_for_dirac} from Theorem \ref{MNT_main_theorem} using an appropriate transformation of Dirac differential equation \eqref{differential_equation} to Krein system \eqref{diff_system}. To do that, we need some notation. Consider system \eqref{diff_system} and let the functions $\phi, \psi, p,q$ be as follows:
\begin{align}
\label{phi_transform}
    \phi(r, \lambda) &= \frac{\exp(-i\lambda r)}{2}\left[P(2r,\lambda) + P_*(2r,\lambda)\right], \quad\phi(0,\lambda) = 1,
    \\ 
    \label{psi_transform}
    \psi(r, \lambda) &= \frac{\exp(-i\lambda r)}{2i}\left[P(2r,\lambda) - P_*(2r,\lambda)\right],\quad\psi(0,\lambda) = 0,
    \\ \nonumber
    p(r) &= -2\Re a(2r),\qquad q(r)  = 2\Im a(2r).
\end{align}
Direct calculations give that functions $\phi$ and $\psi$ are the solutions of the differential equation
\begin{align}\label{transformed_equation}
   \lambda
   \begin{pmatrix}
       \phi\\
       \psi
   \end{pmatrix} 
   =
   J
   \begin{pmatrix}
       \phi\\
       \psi
   \end{pmatrix} '
   + 
   Q
   \begin{pmatrix}
       \phi\\
       \psi
   \end{pmatrix} ,
   \, \phi(0,\lambda) = 1,\,\psi(0,\lambda) = 0,
   \\
   \nonumber
   J = 
    \begin{pmatrix}
        0& 1\\
        -1& 0
    \end{pmatrix},
    \quad
    Q =
    \begin{pmatrix}
        -q& p\\
        p& q  
    \end{pmatrix}.
\end{align}
Now we proceed to the proof of Theorem \ref{cesaro_boundness_for_dirac}. Construct a Krein system  with the coefficient $a(r) = -\frac{1}{2}p(r/2) + \frac{i}{2}q(r/2)$, where $p,q$ are the potentials in  equation \eqref{differential_equation}. Uniqueness theorem and \eqref{transformed_equation} give that functions $\phi$ and $\psi$ defined by \eqref{phi_transform} and \eqref{psi_transform} respectively are the unique solutions of equation \eqref{differential_equation}. Since potentials $p,q$ are in $L^2(\R_+)$, coefficient $a$ is also in $L^2(\R_+)$. Spectral measure of  the Krein system with square-summable coefficient belongs to the Szeg\H{o} class (see Theorem 11.1 in \cite{Denisov2009}). Hence, Theorem \ref{MNT_main_theorem} can be applied and the convergence
\begin{align*}
    \frac{1}{r}\int_0^r |P(s,x)|^2\, ds\to(2\pi \sigma'(x))^{-1},\quad \text{ as }r\to\infty,
\end{align*}
holds for almost every $x\in\R$. Consequently, for almost every $x\in\R$
\begin{align*}
    \sup_{r > 0}\frac{1}{r}\int_0^r |P(s,x)|^2\, ds < \infty.
\end{align*}
Finally,
\begin{align*}
    \sup_{r > 0}\frac{1}{r}\int_0^r |\phi(s,x)|^2ds &=  \sup_{r > 0}\frac{1}{r}\int_0^r \left|\frac{\exp(-ix s)}{2}\left[P(2s,x) + P_*(2s,x)\right]\right|^2ds  \\
    &= \sup_{r > 0}\frac{1}{r}\int_0^r \left|\frac{P(2s,x) + P_*(2s,x)}{2}\right|^2ds 
    \\
    &\le \sup_{r > 0}\frac{1}{r}\int_0^r \frac{\left| P(2s,x)\right|^2 + \left|P_*(2s,x)\right|^2}{2}\,ds 
    \\
    & = \sup_{r > 0}\frac{1}{r}\int_0^r \left| P(2s,x)\right|^2 ds < \infty
\end{align*}
The equality on the last line follows from Corollary \ref{CD_corollary_module} to Christoffel-Darboux formula (see Section \ref{C-D_section} below). Same argument works for $\psi$.

\section{Asymptotic behaviour of Krein system solutions on the real line. Proof of Theorem \ref{MNT_main_theorem}}\label{proof of MNT theorem}
\subsection{Preliminaries}\label{C-D_section}
Christoffel-Darboux formula is one of the most important identities in the theory of orthogonal polynomials on the unit circle, for details see  Theorem 2.2.7 in \cite{Simon} or in Theorem 11.4.2 in \cite{Szego}. We will use its continuous analog.
\begin{lm}[Christoffel-Darboux formula] Let $P$ and $P_*$ be solutions of the Krein system \eqref{diff_system}.The equalities 
\begin{align}\label{CD_formula}
    P(r,\lambda)\ol{P(r,\mu)} - P_*(r,\lambda)\ol{P_*(r,\mu)} = i(\lambda - \ol{\mu})\int_0^r P(s,\lambda)\ol{P(s,\mu)}\,ds,
    \\
    \label{CD_module}
    \left|P_*(r,\lambda)\right|^2 = \left|P(r,\lambda)\right|^2  + 2 \Im(\lambda)\int_0^r \left|P(s,\lambda)\right|^2ds
\end{align}
hold for any two complex numbers $\lambda,\mu$.
\end{lm}
\begin{proof}
Equation \eqref{CD_module} follows from \eqref{CD_formula} by substituting $\mu = \lambda$; \eqref{CD_formula} can be proved by calculation of derivatives in both sides of the equality.
\end{proof}
\begin{corol}\label{CD_corollary_module}
For any $\lambda\in\Cm$ the following is true:
\begin{enumerate}
    \item $|P_*(r,\lambda)| = |P(r,\lambda)|$ if $\lambda\in\R$ \label{CD_cor_1}
    \item $|P_*(r,\lambda)| > |P(r,\lambda)|$ if $\Im(\lambda) > 0$\label{CD_cor_2}
    \item $|P_*(r,\lambda)| < |P(r,\lambda)|$ if $\Im(\lambda) < 0$\label{CD_cor_3}
\end{enumerate}
\end{corol}
\begin{proof}
All three statements easily follow from the second Christoffel-Darboux formula \eqref{CD_module}.
\end{proof}
\begin{corol}\label{no_zeros}
    $P_*(r, \cdot)$ has no zeros in closed half-plane $\ol{\Cm_+}=\{z\colon \Im z \ge 0\}$.
\end{corol}
\begin{proof}
From Part \ref{CD_cor_2} of Corollary \ref{CD_corollary_module} it follows that there are no roots in an open half-plane $\Cm_+$. Hence, to prove Corollary \ref{no_zeros} we need to show the absence of zeros on the real line. Assume the converse. Then there exists $x\in\R$ such that $P_*(r,x) = 0$. Part \ref{CD_cor_1} of Corollary \ref{CD_corollary_module} gives
\begin{align}\label{converse_of_the_boundary}
    P_*(r,x) = P(r,x) = 0.
\end{align}
On the other hand, the functions $P_*(\cdot,x), P(\cdot, x)$ are the solutions of linear differential system \eqref{diff_system} with non-zero initial values. Hence, \eqref{converse_of_the_boundary} cannot hold. This contradiction concludes the proof.
\end{proof}
\begin{knownlemma}
[S. Denisov Lemma 8.5, \cite{Denisov2009}]\label{lemma8.5}
If a measure $\sigma$ belongs to the Szeg\H{o} class and a sequence $r_n$ is such that  $P(r_n,\lambda_0)\to 0$ for some $\lambda_0\in \Cm_+$. Then the convergence $\left|P_*(r_n, \lambda)\right|\to\left|\Pi(\lambda)\right|$ holds uniformly in $\Cm_+$. 
\end{knownlemma}
\begin{corol}\label{cons_module_PI}
From equation \eqref{CD_module} and Lemma \ref{lemma8.5} we have
\begin{align*}
    |\Pi(\lambda)|^2 = 2 \Im(\lambda)\int_0^{\infty} \left|P(s,\lambda)\right|^2ds.
\end{align*}
\end{corol}
\begin{rem}
A similar relation in the theory of the orthogonal polynomials is given in Chapter 2.4 in \cite{Simon}.
\end{rem}
\subsection{An equivalent form of Theorem \ref{MNT_main_theorem}}
In this section we formulate Theorem \ref{MNT_main_theorem} in a form similar to original Mate-Nevai-Totik Theorem \ref{MNT_main}. By $PW_r$ denote the Paley-Wiener space with a spectrum in $[0,r]$, that is, a space of entire functions $f$ that can be represented as
\begin{align*}
    f(x) = \int_0^r\phi(s)e^{ixs}\,ds,\quad\phi\in L^2[0,r].
\end{align*}
Let a function $m_r$ be given by
\begin{align}\label{extremal_function}
    m_r(\sigma, z) &= \inf\left\{\frac{1}{|f(z)|^2}\int_{-\infty}^{\infty}|f(t)|^2\, d\sigma(t) \text{ }\bigg|\text{ }  f\in PW_r, f(z) \neq 0\right\},\quad z\in\Cm.
\end{align}
Similarly to \eqref{w_n_representation}, the function  $m_r$ can be calculated using the continuous analogs of the orthogonal polynomials. This is shown in the following lemma.
\begin{knownlemma}[S. Denisov, Lemma 8.2, \cite{Denisov2009}]\label{denisov m_a minimizer}
For any $z_0\in\Cm$ the following is true:
\begin{align*}
    m_r(\sigma,z_0) = K_r(z_0,z_0)^{-1}, \quad \text{ where } K_r(z',z) = \int_0^r \ol{P(s,z')}P(s,z) \, ds.
\end{align*}
Infimum in \eqref{extremal_function} is achieved by function
\begin{align*}
   f_r(z) &= \frac{K_r(z_0,z)}{K_r(z_0,z_0)},
\end{align*}
in particular, function $f_r$ itself belongs to  $PW_r$. 
\end{knownlemma}
\noindent Lemma gives
\begin{align*}
    m_r(\sigma, z) &= K_r(z,z)^{-1} = \left(\int_0^r |P(s,z)|^2 \,ds\right)^{-1},\\
     r m_r(\sigma,z) &= \left(\frac{1}{r}\int_0^r |P(s,z)|^2 \,ds\right)^{-1}.
\end{align*}
Substituting the latter equality into Theorem 
\ref{MNT_main_theorem} we get its equivalent form.
\begin{thm}[Equivalent form of Theorem \ref{MNT_main_theorem}]\label{equiv_theorem_polynomials_on_the boundary}
Consider a Krein system with a coefficient in $L^1_{\loc}(\R_+)$ such that its spectral measure $\sigma$ is in the Szeg\H{o} class on the real line. Then for almost all $z\in\R$:
\begin{align*}
\lim_{r\to\infty} r m_r(\sigma,z)= 2\pi \sigma'(z).
\end{align*}
\end{thm}

We prove Theorem \ref{equiv_theorem_polynomials_on_the boundary} by establishing two inequalities 
\begin{align}
    \label{lower_boundary_ineq}
    \limsup r m_r(\sigma,z)&\le 2\pi \sigma'(z),
    \\
    \label{upper_boundary_ineq}
    \liminf r m_r(\sigma,z)&\ge 2\pi \sigma'(z).
\end{align}
in Sections \ref{upper_bound} и \ref{lower_bound} respectively.

\subsection{Upper bound}\label{upper_bound}
\begin{lemma}\label{easy_part_of_theorem}
Assume that the measure $\sigma$ on the real line satisfies condition \eqref{first_condition_Szego_class}. Then the inequality
\begin{align*}
    \limsup r m_r(\sigma,z)\le 2\pi \sigma'(z) 
\end{align*}
holds for almost every $z\in \R$.
\begin{rem}
In the lemma we do not require $\sigma$ to belong to the Szeg\H{o} class. Hence, the inequality  in the lemma is stronger than \eqref{lower_boundary_ineq}.
\end{rem}
\begin{proof}
Consider functions 
\begin{align*}
    R(x) &= \frac{e^{ix} - 1}{ix} = \mathcal{F}\left(\mathbf{1}_{[0,1]}\right)(x),
    \\
    \Phi(x) &= |R(x)|^2.
\end{align*}
Notice that
\begin{align*}
    \mathcal{F}^{-1}\left(\Phi\right) = \mathcal{F}^{-1}\left(|R|^2\right) = \mathcal{F}^{-1}\left(R\right)*\mathcal{F}^{-1}\left(\ol{R}\right) =    \mathbf{1}_{[0,1]} *\mathbf{1}_{[-1,0]} =
    \begin{cases}
    0, &|x| > 1 \\
    1 - |x|, &|x|\le 1
    \end{cases},
\end{align*}
in other words, $\Phi$ is a continuous Fejer kernel. Approximation identity constructed with function $\Phi$ has powerful convergence properties. We believe the following result is folklore.
\begin{knownthm}[Continuous Fejer kernel property]\label{continuous_fejer_kernel}
The convergence
\begin{align*}
    \left(\Phi_r * \sigma \right)(z)\to 2\pi\sigma'(z), \quad r\to\infty,
\end{align*}
where $\Phi_r(x) = r\Phi(rx)$, holds for almost every $z\in \R$.
\end{knownthm}
Let us show how the lemma follows from Theorem \ref{continuous_fejer_kernel}. The function $f_{r,z}(x)= R(r(x - z))$ is in $PW_r$ and $f_{r,z}(z) = 1$. Hence
\begin{align*}
    r m_r(\sigma, z) \le r\int_{-\infty}^{\infty}|f_{r,z}(x)|^2\,d\sigma(x) = r\int_{-\infty}^{\infty}|R(r(x - z))|^2\,d\sigma(x) = (\Phi_r * \sigma)(z).
\end{align*}
It is only remains to pass to the limit superior on $r\to\infty$ and apply the property of the continuous Fejer kernel to conclude the proof.
\end{proof}
\end{lemma}
\subsection{Decay of functions from Paley-Wiener space}\label{pw_lemma}
In this section we prove result concerning functions of Paley-Wiener space; it is independent of the theory of Krein systems but will be useful to us. Let the elementary  Weierstrass factors $E_n$ be defined by formula
\begin{align*}
    E_n(z) =
    \begin{cases} 
    (1-z) & n=0, \\ (1-z)\exp \left( \frac{z^1}{1}+\frac{z^2}{2}+\cdots+\frac{z^n}{n} \right) & n > 0.
    \end{cases}
\end{align*}
\begin{thrm}[Hadamard factorization theorem]
Let $f$ be an entire function of finite order $\rho\ge 0$ and $\{a_n\}_{n = 0}^{\infty}$ be its non-zero roots with multiplicity. Then $f$ admits a factorization
\begin{align*}
    f(z) = e^{g(z)}z^m \prod_{k = 0}^{\infty}E_d\left(\frac{z}{a_k}\right), 
\end{align*}
where $g$ is a polynomial with degree not greater than $\rho$ and  $d=[\rho]$.
\end{thrm}
\begin{lemma}\label{lemma_with_gamma}
Let $f$ be a function from $PW_r$. Suppose $f(0)\neq 0$ and every zero of $f$ is a real number. Then the inequality
\begin{align*}
    |f(t)|\le |f(t + \gamma i)|e^{\gamma r} 
\end{align*}
holds for every $t\in\R$ and $\gamma > 0$.
\begin{proof}
Denote roots of $f$ with multiplicity by $\{a_n\}_{n = 0}^{\infty}$. Functions from Paley-Wiener space are of order $1$. Hence, by the Hadamard factorization theorem, it follows that $f$ can be represented as
\begin{align}\label{hadamard_fact_for_PW}
    f(z) = e^{c_1z + c_2}\prod_{k = 0}^{\infty}\left(1-\frac{z}{a_k}\right)e^{\frac{z}{a_k}}.
\end{align}
If $f(t) = 0$ the inequality in lemma is trivial, otherwise we can write
\begin{align}\label{im(c_1)_ineq}
    \nonumber\left|\frac{f(t  +\gamma i)}{f(t)}\right| &= \left|e^{c_1\cdot \gamma i}\right|\prod_{k = 0}^{\infty}\left|\frac{1 - \frac{t + \gamma i}{a_k}}{1 - \frac{t}{a_k}}e^{\frac{\gamma i}{a_k}}\right| = \left|e^{c_1\cdot \gamma i}\right| \prod_{k = 0}^{\infty}\left|\frac{a_k - t -\gamma i}{a_k - t}\right| 
    \\
    &= \left|e^{c_1\cdot \gamma i}\right| \prod_{k = 0}^{\infty}\left|1 - \frac{\gamma i}{a_k - t}\right| = e^{-\Im(c_1)\gamma}\prod_{k = 0}^{\infty}\left|1 + \frac{\gamma^2}{(a_k - t)^2}\right|^{1/2} \ge e^{-\Im(c_1)\gamma}.
\end{align}
Choose a large real number $N$ and substitute $-iN$ for $z$ into \eqref{hadamard_fact_for_PW}.
\begin{align*}
    |f(-iN)| &= \left|e^{-c_1iN + c_2}\prod_{k = 0}^{\infty}\left(1-\frac{-iN}{a_k}\right)e^{\frac{-iN}{a_k}}\right| 
    \\
    &= e^{N \Im c_1 + \Re c_2 }\prod_{k = 0}^{\infty}\sqrt{1 + \frac{N^2}{a_k^2}}\ge e^{N \Im c_1  + \Re c_2}.
\end{align*}
Since $f$ belongs to $PW_r$, its exponential type is not greater than $r$. So 
\begin{align*}
    r \ge \limsup_{N\to\infty}\frac{\log |f(-iN)|}{N}\ge \limsup_{N\to\infty} \frac{N\Im c_1  +\Re c_2 }{N} = \Im c_1.
\end{align*}
The latter inequality together with inequality $\eqref{im(c_1)_ineq}$ completes the proof.
\end{proof}
\end{lemma}

\subsection{Lower bound}\label{lower_bound}
Throughout this section we suppose that the measure $\sigma$ belongs to the Szeg\H{o} class on the real line. Consider the Szeg\H{o} function $D$ defined by the equality
\begin{align}\label{szego_function_definition}
    D(z)  = \left(\sqrt{2\pi}\Pi(z)\right)^{-1}.
\end{align}
Now if we recall \eqref{pi_modulus_on_the_boundary}, we get that
\begin{align}\label{D module on the boundary}
    |D(z)|^2 &= \sigma'(z)
\end{align}
holds for almost all $z \in \R$. Taking this into account, we see that inequality \eqref{upper_boundary_ineq} is equivalent to
\begin{align}\label{infimum_inequality_with_d}
    \liminf rm_r(\sigma, z)\ge 2\pi|D(z)|^2.
\end{align}
The singular part of $\sigma$ only increases the left side in \eqref{infimum_inequality_with_d} so it is suffices to show \eqref{infimum_inequality_with_d} only for absolutely continuous measures $\sigma$. Hence, without loss of generality, we may assume that $\sigma$ is absolutely continuous with respect to the Lebesgue measure and $d\sigma(x) = |D(x)|^2dx.$

Next, let us introduce some constraints for the point $z$. First, assume that $z$ is a Lebesgue point of  $D$, i.e.,
\begin{align}\label{0_is_lebesgue_point}
    \frac{1}{h}\ilim_{|x - z| < h}|D(x) - D(z)|\,dx\to 0,\quad h\to 0.
\end{align}
Secondly, assume that $D$ has non-tangential boundary values in $z$.
Both of this constraints are satisfied on a set of a full Lebesgue measure; without loss of generality, we may assume that $z$ satisfies the constraints and $z= 0$.   Let us remember  that $f_r$ is a function for which a value of $m_r(\sigma,0)$ is achieved (see Lemma \ref{denisov m_a minimizer}). That means
\begin{align}\label{integral_f_r(x)D(x)}
    \ilim_{\R}\left|f_r(x)\right|^2 d\sigma(x) = \ilim_{\R}\left|f_r(x)D(x)\right|^2 dx = m_r(0)|f_r(0)|^2.
\end{align}
 Fix $\eps > 0$. To prove \eqref{infimum_inequality_with_d}, we need to show that for a large enough $r$
\begin{align}\label{most_wanted}
    rm_r(\sigma,0)>(2\pi- 100\eps)|D(0)|^2 .
\end{align}
In addition,  we may assume that
\begin{align}\label{trivial_condition}
    rm_r(\sigma,0)<2\pi|D(0)|^2,
\end{align}
otherwise there is nothing left to prove.
\subsubsection{Supporting lemmas}
\begin{lemma}\label{lemma_with_gamma_for_fa}
Let $f_r$ be a function defined in Lemma \ref{denisov m_a minimizer}.  For almost every $t\in\R$ and $\gamma > 0$ 
\begin{align*}
    |f_r(t)|\le |f_r(t + \gamma i)|e^{\gamma r}.
\end{align*}
\end{lemma}
\begin{proof}
It is enough to show that $f_r$ satisfies all of the conditions of Lemma $\ref{lemma_with_gamma}$.
Recall that by definition 
\begin{align*}
    f_r(\lambda) &= \frac{K_r(0,\lambda)}{K_r(0,0)},\text{ where } K_r(\lambda',\lambda) = \int_0^r \ol{P(s,\lambda')}P(s,\lambda)\,ds.
\end{align*}
From Lemma \ref{denisov m_a minimizer} we already know that $f_r$ belongs to the Paley-Wiener space $PW_r$; obviously $f_r(0) = 1$. Hence, the only nontrivial question is the absence of zeros outside of the real line. Assume the converse. Then there exists $\lambda\notin \R$ such that $f_r(\lambda) = 0$. Therefore, 
\begin{align*}
    K_r(0, \lambda) =\int_0^r \ol{P(s,0)}P(s,\lambda) \,ds = 0
\end{align*}
and the Christoffel-Darboux formula \eqref{CD_formula} gives
\begin{align*}
    P(r,\lambda)\ol{P(r,0)} - P_*(r,\lambda)\ol{P_*(r,0)} = 0,\\
    \left|P(r,\lambda)P(r,0)\right| = |P_*(r,\lambda)P_*(r,0)|.
\end{align*}
The latter equality contradicts Corollary \ref{CD_corollary_module} of the Christoffel-Darboux formula.
\end{proof}
\begin{lemma}\label{function_tends_to_zero}
For every point $z\in\Cm_+$ the following inequality holds:
\begin{align*}
    |f_r(z)D(z)|\le K_r(0,0)^{-\frac{1}{2}}(4\pi \Im(z))^{-\frac{1}{2}},
\end{align*}
in particular, for any $c > 0$ the function $f_rD$ is bounded in the half-plane $\{z\colon \Im(z) > c\}$.
\end{lemma}
\begin{proof}
\begin{align*}
    | f_r(z)| &= \left| \frac{K_r(0,z)}{K_r(0,0)}\right| =  \left| \frac{\int_0^r \ol{P(s,0)}P(s,z) \,ds}{\int_0^r \left|P(s,0)\right|^2 ds}\right|\le \frac{\left(\int_0^r \left|P(s,0)\right|^2 ds\right)^{\frac{1}{2}}\left(\int_0^r \left|P(s,z)\right|^2 ds\right)^{\frac{1}{2}}}{\int_0^r \left|P(s,0)\right|^2 ds} \le 
    \\
    &\le K_r(0,0)^{-\frac{1}{2}}\left(\frac{|\Pi(z)|^2}{2\Im(z)}  \right)^{\frac{1}{2}} \stackrel{\eqref{szego_function_definition}}{=} K_r(0,0)^{-\frac{1}{2}}\left(\frac{1}{2\Im(z)\cdot 2\pi |D(z)|^2}  \right)^{\frac{1}{2}}
    \\
    &= K_r(0,0)^{-\frac{1}{2}}\left(4\pi \Im(z)  \right)^{-\frac{1}{2}} |D(z)|^{-1}.
\end{align*}
The first inequality on the second line follows from Corollary \ref{cons_module_PI}. Multiplying the last inequality by $|D(z)|$, we obtain the required.
\end{proof}
\begin{lemma}\label{possibility_of_poisson}
For every $z\in \Cm_+$ the following holds: 
\begin{align}\label{kind_of_poisson}
    \frac{1}{2\pi i}\int_{\R}\frac{f_r(x)D(x)}{x - z}\,dx &= f_r(z)D(z),
    \\ \nonumber
    \frac{1}{2\pi i}\int_{\R}\frac{f_r(x)D(x)}{x + z}\,dx &= 0.
\end{align}
\end{lemma}
\begin{rem}
    If the function $f_r D$ was in the Hardy space $H^p$ for some $p\ge 1$, the statement of the lemma would be a well-known fact (see p. 116 in \cite{Koosis}). It is not our case but function $f_rD$ satisfies a ``boundedness'' condition from Lemma \ref{function_tends_to_zero}.
\end{rem}
\begin{proof}
We have already mentioned that the function $\frac{c\Pi(z)^{-1}}{z + i} = \frac{D(z)}{z + i}$ is in the Hardy space $H^2$. Hence, for $z\in \Cm_+$ the function $\frac{D(x)}{x + z}$ also belongs to $H^2$. The function $f_r$ belongs to the Paley-Wiener space and consequently to $H^2$. Therefore both integrals in the lemma are absolutely convergent and define analytic functions in $\Cm_+$.  Moreover, $f_r(x)\frac{D(x)}{x + z}$ is a multiplication of two functions from $H^2$ and consequently it belongs to $H^1$. Integral of $H^1$ function over real line equals $0$ (see p. 122 in \cite{Koosis} or Lemma 3.7 in \cite{Garnett}). Thus, the second equality in lemma is proved.

From Lemma \ref{function_tends_to_zero} it follows that the function $G_{\eps}(z) = f_r(z + \eps i)D(z + \eps i)$ is bounded in the upper half-plane $\Cm_+$. Proceeding as above we deduce that for all $z\in \Cm_+$ 
\begin{align*}
    \frac{1}{2\pi i}\int_{\R}\frac{G_{\eps}(x)}{x + z}\,dx &= 0.
\end{align*}
Next,  the following chain of equalities is true:
\begin{align}\label{better_convergence}
    \nonumber
    \frac{1}{2\pi i}\int_{\R}\frac{f_r(x + \eps i)D(x + \eps i)}{x - z}\,dx &= \frac{1}{2\pi i}\int_{\R}\frac{G_{\eps}(x)}{x - z}\,dx 
    \\ 
    =\frac{1}{2\pi i}\int_{\R}\frac{G_{\eps}(x)}{x - z}\,dx&-  \frac{1}{2\pi i}\int_{\R}\frac{G_{\eps}(x)}{x + z}\,dx = \frac{1}{2\pi i}\int_{\R}\frac{2zG_{\eps}(x)}{x^2  -z^2}\,dx. 
\end{align}
The latter integral can be calculated by the calculus of residues. Take a large positive number $R$ and consider a circuit  $C_R$ consisting of the segment $[-R, R]$ and the half-circle $\Gamma_R$ connecting points $R$ and $-R$ in the upper half-plane.
\begin{align}\label{G_eps_into_integrals}
    \nonumber G_{\eps}(z) &= \text{Res}_z\left(\frac{2zG_{\eps}(x)}{x^2  -z^2}\right) = \frac{1}{2\pi i}\ilim_{C_R}\frac{2zG_{\eps}(x)}{x^2  -z^2}dx 
    \\ 
    &=\frac{1}{2\pi i}\ilim_{[-R, R]}\frac{2zG_{\eps}(x)}{x^2  -z^2}dx + \frac{1}{2\pi i}\ilim_{\Gamma_R}\frac{2zG_{\eps}(x)}{x^2  -z^2}dx.
\end{align}
The first term in the latter sum has a limit as $R\to\infty$.
\begin{align*}
    \lim_{R\to\infty} \frac{1}{2\pi i}\ilim_{[-R, R]}\frac{2zG_{\eps}(x)}{x^2  -z^2}\,dx = \frac{1}{2\pi i}\ilim_{\R}\frac{2zG_{\eps}(x)}{x^2  -z^2}\,dx.
\end{align*}
Since the function $G_{\eps}$ is bounded, the second term tends to zero as $R\to\infty$:
\begin{align*}
    \left|\frac{1}{2\pi i}\int_{\Gamma_R}\frac{2zG_{\eps}(x)}{x^2  -z^2}\,dx\right| \le \frac{1}{2\pi}\int_0^{\pi}\left|\frac{2zG_{\eps}(Re^{i\theta})}{R^2e^{2i\theta} - z^2}\right|R\, d\theta \le \frac{R|z|}{R^2 - |z|^2}\sup\nolimits_{\theta}|G_{\eps}(Re^{i\theta})| =O\left(\frac{1}{R}\right).
\end{align*}
We can therefore pass to the limit $R\to\infty$ in the equality \eqref{G_eps_into_integrals} and obtain
\begin{align*}
     \frac{1}{2\pi i}\int_{\R}\frac{2zG_{\eps}(x)}{x^2  -z^2}\,dx = G_{\eps}(z).
\end{align*}
Substituting this into \eqref{better_convergence} we get 
\begin{align*}
    \frac{1}{2\pi i}\int_{\R}\frac{f_r(x + \eps i)D(x + \eps i )}{x - z}\,dx &= G_{\eps}(z)= f_r(z + i\eps)D(z + i\eps).
\end{align*}
To prove equality \eqref{kind_of_poisson} it is only remains to pass to the limit as $\eps$ tends to $0$. The right part of the equality tends to $f_r(z)D(z)$ so we need to evaluate how the left side differs from the integral in the lemma.
\begin{gather*}
    \left|\int_{\R}\frac{f_r(x + \eps i)D(x + \eps i )}{x - z}\,dx - \int_{\R}\frac{f_r(x  )D(x )}{x - z}\,dx\right|=
    \\
    =
    \left|\int_{\R}f_r(x +i\eps )\frac{D(x + \eps i ) - D(x)}{x - z}\,dx + \int_{\R}D(x)\frac{f_r(x +i\eps )- f_r(x)}{x - z}\,dx\right|\le
    \\
    \le \int_{\R}\left|f_r(x +i\eps )\frac{D(x + \eps i ) - D(x)}{x - z}\right|\,dx + \int_{\R}\left|D(x)\frac{f_r(x +i\eps )- f_r(x)}{x - z}\right|\,dx\le
    \\
    \le\norm[f_r(x + \eps i)]_{2}\cdot\norm[\frac{x + i}{x - z}]_{\infty}\norm[\frac{D(x + \eps i) - D(x)}{x + i}]_{2} 
    +\norm[f_r(x + i\eps) - f_r(x)]_{2}\cdot\norm[\frac{x + i}{x - z}]_{\infty}\norm[\frac{D(x)}{x + i}]_{2}
    \\
    \le C\left(\norm[\frac{D(x + \eps i) - D(x)}{x + i}]_{2}  + \norm[f_r(x + i\eps) - f_r(x)]_{2}\right)\le
    \\
    \le C\left(\norm[ \frac{D(x + \eps i )}{x + i}- \frac{D(x + \eps i)}{x + \eps i + i}]_{2} + \norm[\frac{D(x + \eps i)}{x + \eps i + i} - \frac{D(x)}{x + i}]_{2}  + \norm[f_r(x + i\eps) - f_r(x)]_{2}\right)
\end{gather*}
The first term can be bounded in the following way:
\begin{align*}
    \norm[ \frac{D(x + \eps i )}{x + i}- \frac{D(x + \eps i)}{x + \eps i + i}]_{2} = \eps \norm[\frac{D(x + \eps i)}{(x + \eps i + i)(x + i)}]_{2} \le \eps\norm[\frac{D(x)}{x + i}]_{2}\to 0 .
\end{align*}
From a general property of functions in $H^2$ (see Theorem 3.1 in \cite{Garnett} or  Section 19.2 in  \cite{Levin}) it follows that the second and the third terms in the bracket tend to zero as $\eps\to 0$.
\end{proof}
\begin{lemma}\label{lemma 8}
 For a fixed number $b$ an inequality 
\begin{align*}
    |f_r(t)| \le 2\sqrt{2}e\pi |f_r(0)| <  25 |f_r(0)|
\end{align*}
holds for any large enough $r$ and for all $t\in[-\frac{b}{r}, \frac{b}{r}]$.
\end{lemma}
\begin{proof}
Point $t + i/r$ lies in a cone $K = \{z\colon b|\Im(z)| \ge \Re(z)\}$. By assumption, $D$ has a non-tangential boundary value at the point $0$ and $D(0) \neq 0$. Hence, $D$ is continuous at the point $0$ inside cone $K$ and for a large enough $r$ we have
\begin{align}\label{ineq_for_d_near_boundary}
    \left|D\left(t + \frac{i}{r}\right)\right| > \frac{|D(0)|}{2}.
\end{align}
Lemma \ref{possibility_of_poisson} gives
\begin{align*}
    \left|f_r\left(t + \frac{i}{r}\right)D\left(t + \frac{i}{r}\right)\right|^2 &= \left|\int_{\R}\frac{f_r(x)D(x)}{x -\left(t+  \frac{i}{r}\right)}\,dx\right|^2\le\\
    &\le \int_{\R}\left|f_r(x)D(x)\right|^2 \,dx\int_{\R}\frac{1}{x^2 + \frac{1}{r^2}}\,dx= \\
    &\stackrel{\eqref{integral_f_r(x)D(x)}}{=} m_r(0)|f_r(0)|^2\cdot \pi r \stackrel{\eqref{trivial_condition}}{\le}  2\pi^2|D(0)f_r(0)|^2.
\end{align*}
Hence, 
\begin{align*}
    |f_r(0)| > \frac{1}{\sqrt{2}\pi}\frac{\left|D\left(t + \frac{i}{r}\right)\right|}{|D(0)|}\left|f_r\left(t + \frac{i}{r}\right)\right|.
\end{align*}
Applying Lemma \ref{lemma_with_gamma_for_fa} with $\gamma = \frac{1}{r}$ to bound $\left|f_r\left(t + \frac{i}{r}\right)\right|$ and inequality \eqref{ineq_for_d_near_boundary} we get
\begin{align*}
    |f_r(0)| > \frac{1}{\sqrt{2}\pi}\frac{1}{2} e^{-1}|f_r(t)| >\frac{1}{25}|f_r(t)|.
\end{align*}
This concludes the proof.
\end{proof}
\subsubsection{End of the proof}
\noindent Recall that our aim is the proof of  \eqref{most_wanted}. Let $H_r$ be a function defined by
\begin{align*}
    H_r(x) = \frac{e^{-irx} -1 }{x} .
\end{align*}
\begin{st}
For every positive $\delta$ the following holds:
\begin{align}
    \ilim_{-\infty}^{\infty}\left|H_r(x - \delta i)\right|^2\,dx < 2\pi r.
\end{align}
\begin{proof}
To do the calculations recall the formula of the Fourier transform of $\frac{1}{x^2 + 1}$.
\begin{align*}
    \ilim_{-\infty}^{\infty}\frac{e^{iwx}}{x^2 + 1}\,dx = \pi e^{-|w|}.
\end{align*}
With that in mind, the integral can be easily evaluated.
\begin{align*}
     \ilim_{-\infty}^{\infty}\left|H_r(x - \delta i)\right|^2\,dx &=  \ilim_{-\infty}^{\infty}\frac{\left|e^{-r\delta}e^{irx} - 1\right|^2}{x^2 + \delta^2}\,dx = \ilim_{-\infty}^{\infty}\frac{e^{-2r\delta} + 1 - e^{-r\delta}(e^{irx} + e^{-irx})}{x^2 + \delta^2} \,dx
     \\
     &=\frac{1}{\delta}\left[(e^{-2r\delta} + 1)\ilim_{-\infty}^{\infty}\frac{\,dy}{y^2 + 1} - e^{-r\delta}\ilim_{-\infty}^{\infty}\frac{e^{ir\delta y} + e^{-ir\delta y}}{y^2 + 1}\,dy\right] 
     \\
     &=\frac{1}{\delta}\left[(e^{-2r\delta} + 1)\pi - 2e^{-r\delta}\pi e^{-r\delta}\right] = \pi\frac{1 - e^{-2r\delta}}{\delta} < 2\pi r
\end{align*}
\end{proof}
\end{st}
Put $\delta = \frac{\eps}{r}$. The following inequalities hold:
\begin{align*}
    \left|\ilim_{-\infty}^{\infty} f_r(x)H_r(x - \delta i) D(x)\,dx\right|^2 &\le \ilim_{-\infty}^{\infty}\left|f_r(x) D(x)\right|^2\,dx \ilim_{-\infty}^{\infty}\left|H_r(x - \delta i)\right|^2\,dx
    \\
    &\stackrel{\eqref{integral_f_r(x)D(x)}}{=} m_r(0)|f_r(0)|^2\ilim_{-\infty}^{\infty}\left|H_r(x - \delta i)\right|^2\,dx\le
    \\
    & \le m_r(0)|f_r(0)|^2 2\pi r.
\end{align*}
Hence,
\begin{align}\label{first_ineq_with_H}
    \left|\ilim_{-\infty}^{\infty}f_r(x) H_r(x - \delta i) D(x)\,dx\right|&\le|f_r(0)|\sqrt{2\pi r m_r(0)}
\end{align}
Consider quantities $I, I_1, I_2$, defined by the equality
\begin{align*}
    I = \ilim_{-\infty}^{\infty}f_r(x) H_r(x - \delta i) D(x)\,dx& = \ilim_{-\infty}^{\infty} \frac{e^{-ir(x - \delta i)}f_r(x)D(x)}{x -  \delta i} \,dx - \ilim_{-\infty}^{\infty} \frac{f_r(x)D(x)}{x - \delta i} \,dx = e^{-r\delta}I_2 - I_1.
\end{align*}
Lemma  \ref{possibility_of_poisson} gives
\begin{align}\label{calculated_i1}
    I_1 &= \ilim_{-\infty}^{\infty} \frac{f_r(x)D(x)}{x - \delta i} dx = 2\pi i f_r(\delta i)D(\delta i).
\end{align}
Our further goal will be an estimation 
\begin{align}\label{needed_ineq_for_i_2}
    |I_2| < 4\eps f_r(0)D(0).
\end{align}
Before proving this let us show how \eqref{needed_ineq_for_i_2} can be used to establish inequality \eqref{most_wanted}. Notice that
\begin{align*}
    |I| = |e^{-r\delta}I_2 - I_1|\ge |I_1| - e^{-r\delta}|I_2| \ge |I_1| - |I_2|.
\end{align*}
Substitution of \eqref{calculated_i1},
 \eqref{first_ineq_with_H}, \eqref{needed_ineq_for_i_2} into the last inequality gives
\begin{align}\label{ineq_before_the_end}
    |f_r(0)|\sqrt{2\pi r m_r(0)} \ge 2\pi|f_r(\delta i)D(\delta i)| - 4\eps |f_r(0)D(0)|.
\end{align}
Since the function $D$ has non-tangential limit at the point $0$,  for a large enough $r$
\begin{align}\label{21}
    |D(i\delta)| =  \left|D\left(i\frac{\eps}{r}\right)\right|> (1 - \eps)|D(0)|.
\end{align}
Moreover, from Lemma \ref{lemma_with_gamma_for_fa} we know
\begin{align}\label{22}
    |f_r(\delta i)| \ge e^{-r\delta}|f_r(0)| = e^{-\eps}|f_r(0)|\ge (1 -\eps)|f_r(0)|.
\end{align}
Finally, the substitution of \eqref{21} and \eqref{22} into  inequality \eqref{ineq_before_the_end} finishes the proof:
\begin{align*}
    |f_r(0)|\sqrt{2\pi r m_r(0)} &\ge 2\pi (1 - \eps)(1-\eps)|f_r(0)D(0)| - 4\eps|f_r(0)D(0)|  \ge \\
    &\ge (2\pi - (4 + 4\pi)\eps + 2\pi\eps^2)|f_r(0)D(0)| > (2\pi - 20\eps)|f_r(0)D(0)|,\\
    r m_r(0)&\ge \frac{(2\pi - 20\eps)^2}{2\pi}|D(0)|^2 > (2\pi - 100\eps)|D(0)|^2.
\end{align*}
Thus, it is only remains to prove estimate \eqref{needed_ineq_for_i_2}. 
Consider one more integral $\Tilde{I_2}$.
\begin{align*}
    \Tilde{I_2} = \ilim_{-\infty}^{\infty} \ol{\left(\frac{e^{-irx}f_r(x)}{x -  \delta i}\right)}D(x) \,dx = \ilim_{-\infty}^{\infty} \frac{e^{irx}\ol{f_r(x)}}{x +  \delta i}D(x) \,dx.
\end{align*}
The function $f_r$ is in $PW_r$ so there exists $\phi\in L^2[0, r]$ such that $f_r(x) = \int_0^r \phi(t) e^{itx}dt$. Notice that
\begin{align*}
    e^{irx}\ol{f_r(x)} &= e^{irx}\int_0^r \ol{\phi(t)} e^{-itx}dt = \int_0^r \ol{\phi(t)} e^{i(r - t)x}dt = \int_0^r \ol{\phi(r - t)} e^{itx}dt.
\end{align*}
Hence, the function $e^{irx}\ol{f_r(x)}$ also belongs to $PW_r$ and consequently to the Hardy space $H^2$ in the upper half-plane. As we have mentioned before, the function $\frac{D(x)}{x + \delta i}$ also belong to $H^2$. So 
\begin{align*}
    \frac{e^{irx}\ol{f_r(x)}D(x)}{x + \delta i}\in H^1.
\end{align*}
An integral of a $H^1$ function over the real line equals $0$ (see Lemma 3.7 in \cite{Garnett}), therefore,  $\Tilde{I_2} = 0$.
By definition, put
\begin{align*}
    K_1 = \left\{x\in \R\colon |x| > \frac{b}{r}\right\},\quad K_2 = \R\setminus K_1.
\end{align*}
Split  $I_2$ into two terms and consider them separately.
\begin{align*}
    I_2 = \int_{-\infty}^{\infty} \frac{e^{-irx}f_r(x)}{x -  \delta i}D(x) \,dx =  \int_{K_1} \frac{e^{-irx}f_r(x)}{x -  \delta i}D(x) \,dx + \int_{K_2} \frac{e^{-irx}f_r(x)}{x -  \delta i}D(x) \,dx = I_{21} + I_{22}.
\end{align*}
Similarly, $\Tilde I_2$  is represented as a sum of $\Tilde I_{21}$ and $\Tilde I_{22}$. The following identity is true:
\begin{align}\label{identity_main}
    \nonumber
    I_2/D(0) = I_2/D(0) - \ol{\Tilde I_2/D(0)}&= I_{21}/D(0) + \left(I_{22}/D(0) - \int_{K_2} \frac{e^{-irx}f_r(x)}{x -  \delta i} \,dx\right)  + 
    \\ 
    &- \ol{{\Tilde I_{21}/D(0)}} - \ol{\left(\Tilde I_{22}/D(0) - \int_{K_2}\ol{\left( \frac{e^{-irx}f_r(x)}{x -  \delta i}\right)} \,dx\right)}. 
\end{align}
Let us bound every term in the right side.
\begin{align*}
    |I_{21}|^2 = \left|\int_{K_1 }\frac{e^{-irx }f_r(x)D(x)}{x -  \delta i} \,dx \right|^2&\le \int_{K_1}\left|f_r(x)D(x)\right|^2 dx\int_{K_1 }\left|\frac{e^{-irx}}{x -  \delta i}\right|^2 dx\\
    \int_{K_1}\left|f_r(x)D(x)\right|^2 dx &\le\int_{\R}\left|f_r(x)D(x)\right|^2 dx \stackrel{\eqref{integral_f_r(x)D(x)}}{=} m_r(0)|f_r(0)|^2\\
    \int_{K_1 }\left|\frac{e^{-irx}}{x -  \delta i}\right|^2 dx &=   \ilim_{|x| > \frac{b}{r}}\frac{1}{x^2 +  \delta^2}\,dx \le 2\int_{\frac{b}{r}}^{\infty}\frac{dx}{x^2} = 2r/b.
\end{align*}
Thus, 
\begin{align*}
    |I_{21}|^2\le m_r(0)|f_r(0)|^2\frac{2r}{b} &\le \frac{2}{b}|f_r(0)|^2\left(r m_r(0)\right)\stackrel{\eqref{trivial_condition}}{\le} \frac{4\pi}{b}|f_r(0)D(0)|^2,
    \\
    \frac{|I_{21}|}{|D(0)|}&\le \sqrt{\frac{4\pi}{b}}|f_r(0)|.
\end{align*}
Choose  $b$ so large that $\frac{4\pi}{b}\le \eps^2$ holds. Taking this into account, we obtain 
\begin{align}\label{first_part}
    \frac{|I_{21}|}{|D(0)|}\le \eps|f_r(0)|.
\end{align}
Next, we bound the second term in \eqref{identity_main}. 
\begin{align*}
    \nonumber\left|\frac{I_{22}}{D(0)} - \int_{K_2} \frac{e^{-irx}f_r(x)}{x -  \delta i} \,dx\right| &= \frac{1}{|D(0)|}\left|\int_{K_2} \frac{e^{-irx}f_r(x)}{x -  \delta i}D(x) \,dx - \int_{K_2} \frac{e^{-irx}f_r(x)}{x -  \delta i} D(0)\,dx\right|\le
    \\&\le \frac{1}{|D(0)|}\sup_{K_2}\left|\frac{f_r(x)}{x -\delta i}\right|\int_{K_2} |D(x) - D(0)|\,dx.
\end{align*}
From Lemma $\ref{lemma 8}$ it follows that $|f_r(x)| < 25 |f_r(0)|$ for any $x\in K_2$. Since $0$ is a Lebesgue point of $D$ (inequality \eqref{0_is_lebesgue_point} holds), we can choose $r$ so large that
\begin{align*}
    \ilim_{K_2} |D(x) - D(0)|\,dx = \ilim_{|x| <\frac{b}{r}} |D(x) - D(0)|\,dx <\frac{b}{r} \frac{\eps^2}{25b}|D(0)|.
\end{align*}
From the two latter inequalities we get 
\begin{align}\label{third_part}
    \left|\frac{I_{22}}{D(0)} - \int_{K_2} \frac{e^{-irx}f_r(x)}{x -  \delta i} \,dx\right|&\le \frac{1}{|D(0)|}\cdot\frac{25|f_r(0)|}{\delta}\cdot\frac{b}{r} \frac{\eps^2}{25b}|D(0)| = \frac{\eps^2}{\delta r} |f_r(0)|= \eps |f_r(0)|.
\end{align}
The two remaining terms in \eqref{identity_main} can be bounded in a similar way. Substitution of \eqref{first_part}, \eqref{third_part} and two similar inequalities for $\Tilde I_2$ into identity \eqref{identity_main} establishes required estimate \eqref{needed_ineq_for_i_2}.
\section{Proof of Theorem \ref{Teplyaev_average_convergence}}\label{proof of Teplyaev conjecture}
We begin by proving the average convergence for modules of the continuous orthogonal polynomials.
\begin{lemma}
Assume that the coefficient $a$ of Krein system \eqref{diff_system} belongs to  $L^1_{\loc}(\R_+)$ and the associated spectral measure $\sigma$ belongs to the Szeg\H{o} class on the real line. Then for every $\lambda\in\Cm_+$ we have
\begin{align*}
    \lim_{r\to\infty} \frac{1}{r}\int_0^r\left| P_*(\rho,\lambda)\right|\,d\rho = \left|\Pi(\lambda)\right|.
\end{align*}
\end{lemma}
\begin{proof}
The required equality is equivalent to two inequalities 
\begin{align}
    \label{module_liminf}\liminf_{r\to\infty} \frac{1}{r}\int_0^r\left| P_*(\rho,\lambda)\right|\,d\rho \ge \left|\Pi(\lambda)\right|,
    \\
    \label{module_limsup}\limsup_{r\to\infty} \frac{1}{r}\int_0^r\left| P_*(\rho,\lambda)\right|\,d\rho \le \left|\Pi(\lambda)\right|.
\end{align}
Take the limit inferior for $r$ in Christoffel-Darboux formula \eqref{CD_module}.
\begin{align*}
    \liminf_{r\to\infty} \left| P_*(r,\lambda)\right|^2 &= \liminf_{r\to\infty}\left(\left|P(r,\lambda)\right|^2  + 2 \Im(\lambda)\int_0^r \left|P(s,\lambda)\right|^2ds\right)
\end{align*}
Part (b) of the Krein theorem states that the function $P(\cdot,\lambda)$ belongs to  $ L^2(\R)$. Hence,
\begin{align*}
    \liminf_{r\to\infty} \left|P(r,\lambda)\right|^2 = 0.
\end{align*}
The second term tends to the integral from $0$ to $\infty$; by Corollary \ref{cons_module_PI} this integral is equal to $|\Pi(\lambda)|^2$. Thus,
\begin{align*}
    \liminf_{r\to\infty} \left| P_*(r,\lambda)\right|^2 &= |\Pi(\lambda)|^2,\\
    \liminf_{r\to\infty} \left| P_*(r,\lambda)\right| &= |\Pi(\lambda)|.
\end{align*}
Inequality \eqref{module_liminf} follows from the last equality by the direct integration. 
To prove \eqref{module_limsup},  integrate and divide by $r$  Christoffel-Darboux formula \eqref{CD_module}.
\begin{align*}
    \frac{1}{r}\int_0^r\left| P_*(\rho,\lambda)\right|^2d\rho &= \frac{1}{r}\int_0^r\left( \left|P(\rho,\lambda)\right|^2  + 2 \Im(\lambda)\int_0^{\rho} \left|P(s,\lambda)\right|^2ds \right)d\rho 
    \\ 
    &= \frac{1}{r}\int_0^r \left|P(\rho,\lambda)\right|^2d\rho + 2\Im(\lambda)  \frac{1}{r}\int_0^{r}\int_0^{\rho} \left|P(s,\lambda)\right|^2ds\,d\rho 
    \\
    &\le \frac{1}{r}\int_0^{\infty} \left|P(\rho,\lambda)\right|^2d\rho + 2\Im(\lambda)  \frac{1}{r}\int_0^{r}\int_0^{\infty} \left|P(s,\lambda)\right|^2ds\,d\rho
    \\
    &= \left(\frac{1}{2r\Im(\lambda)}+1 \right) 2\Im(\lambda)  \int_0^{\infty} \left|P(s,\lambda)\right|^2ds \stackrel{\text{Cor.} \ref{cons_module_PI}}{=} \left(\frac{1}{2r\Im(\lambda)}+1 \right) |\Pi(\lambda)|^2
\end{align*}
Now, the Cauchy-Schwarz inequality gives
\begin{align}\label{ineq_for_F_r_lemma}
    \frac{1}{r}\int_0^r\left| P_*(\rho,\lambda)\right|\,d\rho\le \sqrt{\frac{1}{r}\int_0^r\left| P_*(\rho,\lambda)\right|^2d\rho} \le \sqrt{1 + \frac{1}{2r\Im(\lambda)}}|\Pi(\lambda)|.
\end{align}
Passing to the limit superior on $r$ concludes the proof.
\end{proof}
\begin{corol}\label{equality_in_i}
Suppose that the spectral measure $\sigma$ belongs to the Szeg\H{o} class and the function $a$ is real-valued. Then
\begin{align*}
    \lim_{r\to\infty} \frac{1}{r}\int_0^r P_*(\rho,i)\,d\rho = \Pi(i).
\end{align*}
\begin{proof}
Substituting $i$ for $\lambda$ in \eqref{Pi_definition}, we obtain
\begin{align*}
    \Pi(i) = \frac{1}{\sqrt{2\pi}}\exp\left[\frac{1}{2\pi i}\int_{-\infty}^{\infty}\frac{(1 + si)\ln \sigma'(s)}{(i - s)(1 + s^2)}\,ds\right] = \frac{1}{\sqrt{2\pi}}\exp\left[\frac{-1}{2\pi }\int_{-\infty}^{\infty}\frac{\ln \sigma'(s)}{(1 + s^2)}\,ds\right].
\end{align*}
Hence, $\Pi(i)$  is a positive real number. System \eqref{diff_system} with $\lambda = i$ has real coefficients, real initial condition and consequently the solution $P_*(r,i)$ is real for every $r$. By Corollary \ref{no_zeros}, $P_*(r,i)$ does not equal zero for $r > 0$; $P_*(0,i) > 0$. Therefore, $P_*(r,i)$ must be positive for every $r\ge 0$ and
\begin{align*}
    \lim_{r\to\infty} \frac{1}{r}\int_0^r P_*(\rho,i)\,d\rho = \lim_{r\to\infty} \frac{1}{r}\int_0^r \left|P_*(\rho,i)\right|\,d\rho = \left|\Pi(i)\right| =  \Pi(i).
\end{align*}
\end{proof}
\end{corol}
\noindent Now we can proceed directly to the proof of Theorem \ref{Teplyaev_average_convergence}.
\begin{proof}[Proof of Theorem \ref{Teplyaev_average_convergence}]
Consider a family $\mathcal{F}$ of the analytic functions
\begin{align*}
    F_r(\lambda) = \frac{1}{r}\int_0^r P_*(\rho,\lambda)\,d\rho,\quad r > 1.
\end{align*}
Notice that
\begin{align}\label{F_r_module_inequality}
    |F_r(\lambda)| = \left|\frac{1}{r}\int_0^r P_*(\rho,\lambda)\,d\rho\right| \le \frac{1}{r}\int_0^r \left|P_*(\rho,\lambda)\right|\, d\rho \stackrel{\eqref{ineq_for_F_r_lemma}}{\le} \sqrt{1 + \frac{1}{2r \Im(\lambda)}}|\Pi(\lambda)|.
\end{align}
Consider an arbitrary compact set $D\subset \Cm_+$. $\Im(\lambda)$ is separated from $0$ in $D$ so
\begin{align*}
    \sup_{\lambda\in D, r}|F_r(\lambda)| \le \sup_{\lambda\in D}c(\lambda)|\Pi(\lambda)| < \infty.
\end{align*}
Hence, the family $\mathcal{F}$ is uniformly bounded on compact subsets and, by Montel theorem, every subset of $\mathcal{F}$ has a convergent subsequence.
Let a sequence $r_n\to\infty$ be such that a sequence $F_{r_n}$ is convergent and $\lim F_{r_n} = \tilde F$.
Inequality \eqref{F_r_module_inequality} gives
\begin{align*}
    |\tilde F(\lambda)| = \lim_{n\to\infty} |F_{r_n}(\lambda)|\le|\Pi(\lambda)|
\end{align*}
for every point $\lambda\in\Cm_+$. On the other hand, from corollary \ref{equality_in_i} it follows that
\begin{align*}
    \tilde F(i) = \lim_{n\to\infty} F_{r_n}(i) = \lim_{n\to\infty}\frac{1}{r_n}\int_0^{r_n} P_*(\rho,i)\,d\rho = \Pi(i).
\end{align*}
Applying the maximum modulus principle, we conclude that $\tilde F(\lambda)=\Pi(\lambda)$ for every $\lambda\in\Cm_+$.
Any partial limit of $F_r$ as $r\to\infty$ coincides with  $\Pi$, so $F_r$ has a limit equal to $\Pi$. This completes the proof.

\end{proof}
\section{Acknowledgements}
The author is grateful to Roman Bessonov for helpful discussions and  constant attention to this work.

\bibliographystyle{plain}
\bibliography{references.bib}

\begin{thebibliography}{10}

\bibitem{aptekarev2015}
A.~I. Aptekarev, S.~A. Denisov, and D.~N. Tulyakov.
\newblock V. {A}. {S}teklov's problem of estimating the growth of orthogonal
  polynomials.
\newblock {\em Proc. Steklov Inst. Math.}, 289(1):72--95, 2015.
\newblock Translation of Tr. Mat. Inst. Steklova {{\bf{2}}89} (2015), 83--106.

\bibitem{bessonov2020}
R.~V. Bessonov.
\newblock Szeg{o} condition and scattering for one-dimensional {D}irac
  operators.
\newblock {\em Constr. Approx.}, 51(2):273--302, 2020.

\bibitem{bessonov20201}
Roman Bessonov and Sergey Denisov.
\newblock A spectral {S}zego theorem on the real line.
\newblock {\em Adv. Math.}, 359:106851, 41, 2020.

\bibitem{christ2002}
M.~Christ and A.~Kiselev.
\newblock Scattering and wave operators for one-dimensional {S}chr\"{o}dinger
  operators with slowly decaying nonsmooth potentials.
\newblock {\em Geom. Funct. Anal.}, 12(6):1174--1234, 2002.

\bibitem{christ20012}
Michael Christ and Alexander Kiselev.
\newblock Maximal functions associated to filtrations.
\newblock {\em J. Funct. Anal.}, 179(2):409--425, 2001.

\bibitem{christ20011}
Michael Christ and Alexander Kiselev.
\newblock W{KB} asymptotic behavior of almost all generalized eigenfunctions
  for one-dimensional {S}chr\"{o}dinger operators with slowly decaying
  potentials.
\newblock {\em J. Funct. Anal.}, 179(2):426--447, 2001.

\bibitem{deift1999}
P.~Deift and R.~Killip.
\newblock On the absolutely continuous spectrum of one-dimensional
  {S}chr\"{o}dinger operators with square summable potentials.
\newblock {\em Comm. Math. Phys.}, 203(2):341--347, 1999.

\bibitem{Denisov2002}
S.~A. Denisov.
\newblock To the spectral theory of {K}re\u{\i}n systems.
\newblock {\em Integral Equations Operator Theory}, 42(2):166--173, 2002.

\bibitem{denisov2004}
S.~A. Denisov.
\newblock On the existence of wave operators for some {D}irac operators with
  square summable potential.
\newblock {\em Geom. Funct. Anal.}, 14(3):529--534, 2004.

\bibitem{denisov2018}
S.~A. Denisov.
\newblock On the growth of polynomials orthogonal on the unit circle with a
  weight {$w$} that satisfies {$w,w^{-1}\in L^\infty(\Bbb T)$}.
\newblock {\em Mat. Sb.}, 209(7):71--105, 2018.

\bibitem{denisov2017}
Sergey Denisov and Keith Rush.
\newblock Orthogonal polynomials on the circle for the weight {$w$} satisfying
  conditions {$w, w^{-1}\in{\rm BMO}$}.
\newblock {\em Constr. Approx.}, 46(2):285--303, 2017.

\bibitem{Denisov2009}
Sergey~A. Denisov.
\newblock Continuous analogs of polynomials orthogonal on the unit circle and
  {K}re\u{\i}n systems.
\newblock {\em IMRS Int. Math. Res. Surv.}, pages Art. ID 54517, 148, 2006.

\bibitem{Garnett}
John~B. Garnett.
\newblock {\em Bounded analytic functions}, volume~96 of {\em Pure and Applied
  Mathematics}.
\newblock Academic Press, Inc. [Harcourt Brace Jovanovich, Publishers], New
  York-London, 1981.

\bibitem{Koosis}
Paul Koosis.
\newblock {\em Introduction to {$H_p$} spaces}, volume 115 of {\em Cambridge
  Tracts in Mathematics}.
\newblock Cambridge University Press, Cambridge, second edition, 1998.
\newblock With two appendices by V. P. Havin [Viktor Petrovich Khavin].

\bibitem{Krein}
M.~G. Kre\u{\i}n.
\newblock Continuous analogues of propositions on polynomials orthogonal on the
  unit circle.
\newblock {\em Dokl. Akad. Nauk SSSR (N.S.)}, 105:637--640, 1955.

\bibitem{Levin}
B.~Ya. Levin.
\newblock {\em Lectures on entire functions}, volume 150 of {\em Translations
  of Mathematical Monographs}.
\newblock American Mathematical Society, Providence, RI, 1996.
\newblock In collaboration with and with a preface by Yu. Lyubarskii, M. Sodin
  and V. Tkachenko, Translated from the Russian manuscript by Tkachenko.

\bibitem{Mate_Nevai_Totik}
Attila M\'{a}t\'{e}, Paul Nevai, and Vilmos Totik.
\newblock Szeg{o}'s extremum problem on the unit circle.
\newblock {\em Ann. of Math. (2)}, 134(2):433--453, 1991.

\bibitem{muscalu2003}
Camil Muscalu, Terence Tao, and Christoph Thiele.
\newblock A {C}arleson theorem for a {C}antor group model of the scattering
  transform.
\newblock {\em Nonlinearity}, 16(1):219--246, 2003.

\bibitem{muscalu20032}
Camil Muscalu, Terence Tao, and Christoph Thiele.
\newblock A counterexample to a multilinear endpoint question of {C}hrist and
  {K}iselev.
\newblock {\em Math. Res. Lett.}, 10(2-3):237--246, 2003.

\bibitem{rybalko1966}
A.~M. Rybalko.
\newblock On the theory of continual analogues of orthogonal polynomials.
\newblock {\em Teor. Funkci\u{\i} Funkcional. Anal. i Prilo\v{z}en. Vyp.},
  3:42--60, 1966.

\bibitem{sakhnovich2000}
L.~A. Sakhnovich.
\newblock On the spectral theory of a class of canonical differential systems.
\newblock {\em Funktsional. Anal. i Prilozhen.}, 34(2):50--62, 96, 2000.

\bibitem{simon1982}
Barry Simon.
\newblock Schr\"{o}dinger semigroups.
\newblock {\em Bull. Amer. Math. Soc. (N.S.)}, 7(3):447--526, 1982.

\bibitem{Simon}
Barry Simon.
\newblock {\em Orthogonal polynomials on the unit circle. {P}art 1}, volume~54
  of {\em American Mathematical Society Colloquium Publications}.
\newblock American Mathematical Society, Providence, RI, 2005.
\newblock Classical theory.

\bibitem{Szego}
G\'{a}bor Szeg\H{o}.
\newblock {\em Orthogonal polynomials}.
\newblock American Mathematical Society, Providence, R.I., fourth edition,
  1975.
\newblock American Mathematical Society, Colloquium Publications, Vol. XXIII.

\bibitem{Teplyaev2005}
Alexander Teplyaev.
\newblock A note on the theorems of {M}. {G}. {K}rein and {L}. {A}.
  {S}akhnovich on continuous analogs of orthogonal polynomials on the circle.
\newblock {\em J. Funct. Anal.}, 226(2):257--280, 2005.

\end{thebibliography}

\end{document}